\documentclass[letterpaper, 10 pt, conference]{ieeeconf}  
\usepackage{amsmath}
\usepackage{amssymb}
\IEEEoverridecommandlockouts                              
\usepackage{graphics} 
\usepackage{epsfig} 
\usepackage{epstopdf}
\usepackage{mathptmx} 
\usepackage{times} 
\usepackage{amsmath} 
\usepackage{amssymb}  
\usepackage[linesnumbered,lined,boxed,commentsnumbered]{algorithm2e}
\usepackage[noend]{algpseudocode}
\usepackage{subfigure}

\newtheorem{defn}{Definition}

\newtheorem{prop}{Proposition}
\newtheorem{lem}{Lemma}

\newtheorem{cor}{Corollary}

\newcommand{\mbb}[1]{\mathbb{#1}}
\newcommand{\mbf}[1]{\mathbf{#1}}

\newcommand{\R}{\mathbb{R}}

\newcommand{\bmat}[1]{\begin{bmatrix}#1\end{bmatrix}}

\title{\LARGE \bf
Solving Dynamic Programming with Supremum Terms in the Objective and Application to Optimal Battery Scheduling for Electricity Consumers Subject to Demand Charges
}

\author{Morgan Jones, Matthew M. Peet
}

\begin{document}

\maketitle
\thispagestyle{empty}
\pagestyle{empty}

\begin{abstract}
In this paper, we consider the problem of dynamic programming when supremum terms appear in the objective function. Such terms can represent overhead costs associated with the underlying state variables. Specifically, this form of optimization problem can be used to represent optimal scheduling of batteries such as the Tesla Powerwall for electrical consumers subject to demand charges - a charge based on the maximum rate of electricity consumption. These demand charges reflect the cost to the utility of building and maintaining generating capacity. Unfortunately, we show that dynamic programming problems with supremum terms do not satisfy the principle of optimality. However, we also show that the supremum is a special case of the class of forward separable objective functions. To solve the dynamic programming problem, we propose a general class of optimization problems with forward separable objectives. We then show that for any problem in this class, there exists an augmented-state dynamic programming problem which satisfies the principle of optimality and the solutions to which yield solutions to the original forward separable problem. We further generalize this approach to stochastic dynamic programming problems and apply the results to the problem of optimal battery scheduling with demand charges using a data-based stochastic model for electricity usage and solar generation by the consumer.

%
\end{abstract}

\section{Introduction}

In 2012, 95,000 new distributed solar PhotoVoltaic (PV) systems were installed nationally, a 36\% increase from 2011 and yielding a total of approximately 300,000 installations total~\cite{sepa_release}. Further, utility-scale PV generating capacity has increased at an even faster rate, with 2012 installations more than doubling that of 2011~\cite{irec_solar_2013}. Meanwhile, partially due to the development of energy-efficient appliances and new materials for insulation, US electricity demand has plateaued~\cite{annual_outlook2014}. As a consequence of these trends, utility companies are faced with the problem that demand \textit{peaks} continue to grow. Specifically, as per the US EIA~\cite{EIA_report}, the ratio of peak demand to average demand has increased dramatically over the last 20 years.

Fundamentally, the problem faced by utilities is that consumers are typically charged based on total electricity consumption, while utility costs are based both on consumption and for building and maintaining the generating capacity necessary to meet peak demand. Recently, several public and private utilities have moved to address this imbalance by charging residential consumers based on the maximum \textit{rate} (\$ per kW) of consumption - a cost referred to as a \textit{demand charge}. Specifically, in Arizona, both major utilities SRP and APS have mandatory demand charges for residential consumers~\cite{Price}.

For consumers, load is relatively inflexible and hence the most direct approach to minimizing the effect of demand charges is the use of battery storage devices such as the Tesla Powerwall~\cite{farhangi_2010,mohd_2008,dunn_2011}. These devices allow consumers to shift electricity consumption away from periods of peak demand, thereby minimizing the effect of demand charges. In this paper, we specifically focus on battery storage coupled with HVAC and solar generation. This is due to the fact that load from HVAC and electricity from solar generation can be forecast well apriori.

The use of battery storage has been well documented in the literature \cite{Battery} and in particular, there have been several results on the optimal use of batteries for residential customers~\cite{Many_Names,Japan,example,shaving}. Within this literature, there are relatively few results which include demand charges. Of those which do treat demand charges, we mention~\cite{Malky} which proposes a heuristic form of dynamic programming, and the recent work in~\cite{Braun}, wherein the optimization problem is broken down into several agents, and a Lagrangian approach is used to preform the optimization. Furthermore, in \cite{Vijay} a similar energy storage problem is solved using optimized curtailment and load shedding. An $L_p$ approximation of the demand charge was used in combination with multi-objective optimization in~\cite{multi-objective} and, in addition, the optimal use of building mass for energy storage was considered in~\cite{Thermostat}, wherein a bisection on the demand charges was used. However, we note that none of these approaches resolve the fundamental mathematical problem of dynamic programming with a non-separable cost function and hence are either inaccurate,  computationally expensive, or are not guaranteed to converge. Finally, we note that there has been no work to date on optimization of demand charges coupled with stochastic models of solar generation.

In this paper, we formulate the battery storage problem as a dynamic program with an objective function consisting of both integrated time-of-use charges and a supremum term representing the demand charge. Furthermore, we model solar generation as a Gauss-Markov process and minimize the expected value of the objective.
The fundamental mathematical challenge with dynamic programming problems of this form is that, as shown in Section~\ref{sec:GDP}, problems which include supremum terms in the objective do not satisfy the \textit{principle of optimality} and thus recursive solution of the Bellman equation does not yield an optimal policy.

Dynamic programming for problems which do not satisfy the principle of optimality has received little attention and there are few results in the literature in which this problem has been addressed. The only generalized approach to the problem seems to be that taken in~\cite{Duan} which considered the use of multi-objective optimization in the case where the objective function is ``backward separable''. Although the supremum term is not backward separable, an $L_p$ approximation of the supremum \textit{is} backward separable and this approach was applied in~\cite{multi-objective} to the problem of battery storage. Although not directly addressed in~\cite{Duan}, our approach is inspired by this result and is based on the observation that while the supremum is not backward separable, it \textit{is} ``forward separable''.

To solve forward separable optimization problems, we propose in this paper a rigorous approach to a class of generalized dynamic programming problems which are formally defined in Section~\ref{sec:GDP}. For this class of problems, we propose a precise definition of the \textit{principle of optimality} and show that if this definition holds, then the Bellman equation can be used to define an optimal policy. Next, we propose a class of forward separable optimization problems and show that dynamic programming with integral and supremum terms is an element of this class. We then show that the principle of optimality fails for certain problems in this class. In Section~\ref{sec:ADP}, we show that for any forward separable dynamic programming problem, there exists a separable augmented-state dynamic programming problem for which the principle of optimality holds and from which solutions to the original forward separable problem can be recovered. In Section~\ref{sec:battery1}, we apply these methods to the battery scheduling problem for a given load and solar generation schedule. In Section~\ref{sec:SDP}, we show that the augmented dynamic programming problem can also be used to solve stochastic dynamic programming problems with forward separable objectives and apply this approach to the battery scheduling problem using a Gauss-Markov model of solar generation extracted from data provided by local utility SRP.

  \section{Background: Generalized Dynamic Programming}\label{sec:GDP}
In this paper, we consider a generalized class of dynamic programming problems. Specifically, we define a generalized dynamic programming problem as an indexed sequence of optimization problems $G(t_0,x_0)$, defined by a an indexed sequence of objective functions $J_{t_0,x_0}: \R^{m \times (T-t_0)} \times \R^{n \times (T-t_0+1)}$ where we say that $\mathbf{u}^*\in \R^{m \times (T-t_0)}$ and $\mathbf{x}^*\in \R^{n \times (T-t_0+1)}$ solve $G(t_0,x_0)$ if
 \begin{align}
 &(\mathbf{u}^*,\mathbf{x}^*)= \arg \min_{\mathbf u, \mathbf x} J_{t_0,x_0}(\mathbf u, \mathbf x) \label{eqn:opt}\\ \nonumber
 &\text{subject to:  }  x(t+1)=f[x(t),u(t),t] \text{,  given  } x(t_0)=x_0 \\ \nonumber
    & x(t) \in X \subset \mathbb{R}^n \text{ for  } t=t_{0}+1,..,T \\ \nonumber
    & u(t) \in U \subset \mathbb{R}^m\text{ for  } t=t_0,..,T-1  \nonumber
 \end{align}
 Where $f$ : $\mathbb{R}^n \times \mathbb{R}^m \times \mathbb{N} \to \mathbb{R}^n$, $x(t) \in \mathbb{R}^n$ and $u(t) \in \mathbb{R}^m$ for all $t$.
 We denote $J_{t_0,x_0}^* =J_{t_0,x_0}(\mathbf u^*, \mathbf x^*)$.

\begin{defn}
     	We say the sequence of controls  $\bold{u}=(u({t_0}),....,u({T-1})) \in \R^{m \times (T-t_0)}$ is feasible if $u(t) \in U \text{ for  } t=t_0,..,T-1$ and if $x(t+1)=f[x(t),u(t),t]$ and $x(t_0)=x_0$, then $x(t)\in X$ for all $t$. For a given $x$, we denote by $\Gamma_{t,x}$, the set $u \in U$ such that $f[x,u,t] \in X$. In this paper we only consider problems where $\Gamma_{t,x}$ is nonempty for all $x$ and $t$.
\end{defn}
Note that for this class of optimization problems, feasibility is inherited. That is, if $\bold{u}=(u({t}),....,u({T-1}))$ and $\mathbf{x} =(x(t),\cdots,x(T))$ are feasible for $G(t,x(t))$ and  $\bold{v}=(v({s}),....,v({T-1}))$ and $\mathbf{h} =(h(s),\cdots,h(T))$ are feasible for $G(s,x(s))$ where $s>t$, then $\bold{w}=(u(t),\cdots,u(s-1),v({s}),....,v({T-1}))$ and $\mathbf{z} =(x(t),\cdots,x(s-1),h(s),\cdots,h(T))$ are feasible for $G(t,x(t))$.

In certain cases, indexed optimization problems of the Form of $G(t_0,x_0)$ can be solved using an optimal policy.

\begin{defn}
       	A \textit{policy} is any map from the present state and time to a feasible input $(x,t) \mapsto u(t) \in \Gamma_{x,t}$, as $u(t)=\pi(x,t)$. We say that $\pi^*$ is an \textit{optimal policy} for Problem~\eqref{eqn:opt} if
       \[
\mathbf{u}^* =(\pi^*(x_0,t_0),....,\pi^*(x(T-1),T-1) )
       \]
where $x({t+1})^*=f[x(t)^*,\pi^*(x(t)^*,t),t]$ for all $t$.
\end{defn}

The existence of an optimal policy states that knowledge of the current state is sufficient to determine the current input. Existence of such a policy vastly simplifies the optimization problem. However, not every generalized dynamic programming problem admits an optimal policy. The ``Principle of Optimality'' defines one class of optimization problems for which there exists an optimal policy.

\begin{defn}
We say an optimization problem, $G(t_0,x_0)$, of the Form~(1) \textit{satisfies the principle of optimality} if the following holds. For any $s$ and $t$ with $t_0 \le t<s<T$, if $\mathbf u^*=(u(t),...,u(T-1))$ and $ \mathbf x^*=(x(t),...,x(T))$ solve $G(t,x(t))$ then  $\mathbf v=(u(s),...,u(T-1))$ and $ \mathbf h=(x(s),...,x(T))$ solve $G(s,x(s))$.
\end{defn}

The classical form of Dynamic programming algorithm, as originally defined in~\cite{Bellman}, can be used to solve indexed optimization problems of the Form~\eqref{eqn:opt}. This algorithm has the advantage of computational complexity which is linear in $T$.

Dynamic Programming algorithms are most commonly used to solve the special class of indexed optimization problems $P(t_0,x_0)$ of the form
 \begin{align}
  &\min_{\mathbf u, \mathbf x}  J_{t_0,x_0}(\mathbf u, \mathbf x) = \sum_{t=t_0}^{T-1}{c_{t}(x(t),u(t))} + c_{T}({x(T)}) \label{eqn:DP}\\ \nonumber
  &\text{subject to:  }  x(t+1)=f[x(t),u(t),t] \text{,  given  } x(t_0)=x_0  \\ \nonumber
      & x(t) \in X \text{ for  } t=t_{0}+1,..,T \\ \nonumber
      & u(t) \in U \text{ for  } t=t_0,..,T-1  \nonumber
 \end{align}
 Note that $J_{T,x}=c_{T}({x})$. We will refer to $x(t_0) \in \mathbb{R}^n$ the initial state, $x(t) \in \mathbb{R}^n$ the state at time t and $u(t) \in \mathbb{R}^m$ the inputs at time t. $J_{t_0,x_0}$ is the objective function, $c_t$ : $\mathbb{R}^n \times \mathbb{R}^m \to \mathbb{R}$ for $t=t_0,..,T-1$, $c_T$ $\mathbb{R}^n \to \mathbb{R}$ are given functions and $f$ : $\mathbb{R}^n \times \mathbb{R}^m \times \mathbb{N} \to \mathbb{R}^n$ is a given vector field. The following lemma shows that this class of problems satisfies the principle of optimality.\\
 \begin{lem}
Any Problem of Form $P(t_0,x_0)$ in (2) satisfies the principle of optimality.
 \end{lem}
 \begin{proof}
 	Suppose $\mathbf u^*=(u(t),...,u(T-1))$ and $ \mathbf x^*=(x(t),...,x(T))$ solve $P(t,x(t))$ in (2). Now we suppose by contradiction that there exists some $s >t$ such that $\mathbf v=(u(s),...,u(T-1))$ and $ \mathbf h=(x(s),...,x(T))$ do not solve $P(s,x(s))$. We will show that this implies that $\mathbf u^*$ and $ \mathbf x^*$ do not solve $P(t,x)$ in (2), thus verifying the conditions of the Principle of Optimality.  If $\mathbf v$ and $ \mathbf h$ do not solve $P(s,x(s))$, then there exist feasible $\mathbf{w}$, $\mathbf z$ such that
 $ J_{s,x(s)}(\mathbf v, \mathbf h) < J_{s,x(s)}(\mathbf w, \mathbf z) $. i.e.

 	\begin{align*}
 	&  J_{s,x(s)}(\mathbf w,\mathbf z )\\
 	& = \sum_{t=s}^{T-1}{c_{t}(z(t),w(t))} + c_{T}({z(T)}) \\
 	& < \sum_{t=s}^{T-1}{c_{t}(x(t),u(t))} + c_{T}({x(T)}) \\
 	& =J_{s,x(s)}(\mathbf v,\mathbf h )
 	\end{align*}
 	Now consider the proposed feasible sequences $\hat{\mathbf u}=(u(t),...,u(s-1), w(s),...,w(T-1))$ and $\hat{\mathbf x}=(x(t),...,x(s-1), z(s),...,z(T-1))$. It follows:
 	\begin{align*}
 	& J_{t,x(t)}(\hat{\mathbf u},\hat{\mathbf x} ) \\
 	& =\sum_{k=t}^{s-1}{c_{k}(x(k),u(k))} + \sum_{k=s}^{T-1}{c_{k}(z(k),w(k))} + c_{T}({z(T)}) \\
 	& < \sum_{k=t}^{s-1}{c_{k}(x(k),u(k))}  + \sum_{k=s}^{T-1}{c_{k}(x(k),u(k))} + c_{T}({x(T)}) \\
 	& = J_{t,x(t)}(\mathbf u^*,\mathbf x^* )
 	\end{align*}
 	which contradicts optimality of $\mbf u^*, \mbf x^*$. Therefore, this class of problems satisfies the principle of optimality.
 	\end{proof}

 \begin{prop} \label{prop:bellman}
 	Consider the class of optimization problems $P(t_0,x_0)$ in (2). If we define $F(x,t)= J_{t,x}^*$, then the following hold for for all $x \in X$.
 	  	\begin{align}
 	  	& F(x,t)=\inf_{u \in \Gamma_{t,x}}\{c_t(x,u)+F(f(x,u,t),t+1)\} \label{eqn:bellman} \\
 	  	& \qquad \qquad \qquad \qquad \qquad \qquad \forall t \in \{t_0,..,T-1\} \notag \\
 	  	& F(x,T)=c_T(x) \qquad \forall x \in X \notag
 	  	\end{align}
 \end{prop}
 \begin{proof}

Clearly $F(x,T)=c_T(x)$ for any $x$.\\
Now for any $x \in X$ and $t \in \{t_0,..,T-1\}$, suppose $\mathbf u^*=(u(t),..,u(T-1))$ and $\mathbf x^*=(x(t),..,x(T))$ solve $P(t,x)$. By the principle of optimality $\mathbf{v}=(u(t+1),..,u(T-1))$ and $\mathbf h=(x(t+1),..,x(T))$ solve $P(t+1,x(t+1))$. Therefore
\begin{equation}
F(x(t+1),t+1)=J_{x(t+1),t+1}^*=J_{x(t+1),t+1}(\mathbf{v},\mathbf h).
\end{equation}
We conclude that
\begin{align*}
F(x,t) & = J_{x,t}^* \\
& = J_{x,t}(\mathbf u^*, \mathbf x^*) \\
& = c_{t}(x,u(t)) + J_{x(t+1),t+1}(\mathbf{v},\mathbf h) \\
& = c_{t}(x,u(t)) + F(f(x,u(t),t),t+1) \text{  using (4)} \\
& \ge \inf_{u \in \Gamma_{t,x}}\{c_{t}(x,u) + F(f(x,u,t),t+1)\}
\end{align*}
holds for all $x$ and $t$.

Now we prove $F(x,t)\le \inf_{u \in \Gamma_{t,x}}\{c_{t}(x,u) + F(f(x,u,t),t+1)\}$.
For any $u\in \Gamma_{x,t}$, let $\mbf w_u = \{w(t+1),\cdots,w(T-1)\}$ and $\mbf h_u= \{f(x,u,t),z(t+2),\cdots,z(T)\}$ be feasible for $P(f(x,u,t),t+1)$. Then $\mbf v_u = \{u,w(t+1),\cdots,w(T-1)\}$ and $\mbf z_u= \{x,f(x,u,t),z(t+2),\cdots,z(T)\}$ are feasible for $P(t,x)$. Therefore,
\begin{align*}
F(x,t)&=J^*_{t,x}\le J_{t,x}(\mbf v_u,\mbf h_u)\\
&=c_{t}(x,u) + J_{f(x,u,t),t+1}(\mbf w_u,\mbf z_u)\\
&\le c_{t}(x,u) + F(f(x,u,t),t+1)\\
&\le \inf_u \{ c_{t}(x,u) + F(f(x,u,t),t+1)\}
\end{align*}


 \end{proof}
Note: Equation~\eqref{eqn:bellman} is often referred to as Bellman's Equation and a function $F$ which satisfies Bellman's equation  is often referred to as a ``cost to go'' function. Prop.~\ref{prop:bellman} shows that problems of the Form $P(t_0,x_0)$ admit a solution to Bellman's Equation which in turn indexes the optimal objective to the Problem. Furthermore, for problems  $P(t_0,x_0)$, the solution to Bellman's equation can be obtained recursively backwards in time using a minimization on $u$. When $x$ and $u$ are discrete, the RHS of Eqn.~\ref{eqn:bellman} takes a number of finite values and minimization over these values is trivial. When the variables are continuous, finding a functional form for the minimization step is more challenging. In either case, a solution to Bellman's equation provides a state-feedback law or optimal policy as follows.
\begin{cor}
Consider $P(t_0,x_0)$ in (2). Suppose $F(x,t)$ satisfies Equation~\eqref{eqn:bellman} for $P(t_0,x_0)$, Then
\[
\pi^*(x,t) = \arg \inf_{u \in \Gamma_{t,x}}\{c_t(x(t),u)+F(f(x(t),u,t),t+1)\}.
\]
is an optimal policy.
\end{cor}


\noindent \textbf{Dynamic Programming with Supremum Terms}
In this paper we consider the special class of indexed optimization problem, $S(t_0,x_0)$. In contrast to problems of the form $P(t_0,x_0)$ in~\eqref{eqn:opt}, class $S(t_0,x_0)$ has supremum (or maximum) terms in the objective. Specifically, these problems have the following form. \small{ \begin{align}
   &\min_{ \mathbf u, \mbf x} J_{t_0,x_0} (\mbf u,\mbf x):=  \sum_{t=t_0}^{T-1}{c_{t}(x(t),u(t))} +c_{T}(x(T)) +\sup_{t_0 \le k \le T}{d_t(x({k}))} \label{eqn:S}\\ \nonumber
   &\text{subject to:  }  x(t+1)=f[x(t),u(t),t] \\ \nonumber
   & x(0)=x_0 \text{  given}\\ \nonumber
   & x(t) \in X \text{ for  } t=t_0,..,T \\ \nonumber
   & u(t) \in U \text{ for  } t=t_0,..,T-1  \nonumber
   \end{align}
} \normalsize	
   \begin{lem}
   	The class of optimization problems in~\eqref{eqn:S} does not satisfy the principle of optimality.
   	\end{lem}
   	\begin{proof}
   		We give a counterexample. For $h > 0$, we consider the following problem $S(0,0)$:
    		\begin{align*}
    		&  \min_{\mathbf{u}\in \R^3, \mbf x\in \R^4} \quad {\sum_{t=0}^{2}{c_{t}(u({t}))} + \sup_{0 \le k \le 3}{x(k)}}\\
    			& \text{subject to:  } x(t+1)=x(t)+u(t), \quad x(0)=0 \\
    			& 0 \le x_t \le h, \\
    			& u(t) \in \{-h,0,h\}	
    			\end{align*}  	
Where here we define $c_{0}(u(0))=-u(0),$ $c_{1}(u(1))=u(1),$ $c_{2}(u(2))=-u(2)/2$.\\
Since $\mbf u \in \{-h,0,h\}^3$, there are 27 input sequences, only 8 of which are feasible. In Table~\ref{tab:counter}, we calculate the objective value of each feasible input sequence and deduce the optimal input is $\mbf{u}=(h,-h,h)$. Now suppose we follow this input sequence until $t=2$ yielding $x(2)=0$. Now we examine the problem $S(2,x(2))$.
    		\begin{align*}
    		& \min_{\mathbf{u} \in \R \mbf x \in x }{c_{2}(u({2})) + \sup_{2 \le k \le 3}{x(k)}}\\
    		& \text{subject to:  } x(t+1)=x(t)+u(t), \quad x(2)=0 \\
    		& 0 \le x(t) \le h, \\
    		& u(t) \in \{-h,0,h\}	
    		\end{align*}
    For this sub-problem, there are two feasible inputs: $u(3)\in \{-h,0\}$. Of these, the latter is optimal (objective value $h/2$ vs $0$). Thus we see that although $\mbf u=\{h,-h,h\}$ and $\mbf x=\{0,h,0,h\}$ solve $S(0,0)$,  $\mbf v =\{h\}$ and $\mbf h =\{0,h\}$ do not solve $S(2,0)$.
%
%
   		\end{proof}
   		\begin{table}
   			\centering
   			\caption{This table shows the corresponding cost of each feasible policy used in the counter example in lemma 1}
   			\label{tab:counter}
   			\begin{tabular}{|l|l||l|l|l}
   				\cline{1-4}
   				feasible $\mbf u$  & objective value       & feasible $\mbf u$  & objective value   \\ \cline{1-4}
   				$(0,0,0)$  & 0 &  $(h,0,-h)$    & h/2  &  \\ \cline{1-4}
   				$(0,0,h)$    & h/2       & $(h,0,0)$ & 0  &  \\ \cline{1-4}
   				$(0,h,0)$ & 2h        & $(h,-h,0)$  & -h  &  \\ \cline{1-4}
   				$(0,h,-h)$  & (5/2)h & $(h,-h,h)$    & -(3/2)h  &  \\ \cline{1-4}
   			\end{tabular}
   		\end{table}

 \section{Solution Methodology: Augmented Dynamic Programming}\label{sec:ADP}

 In this section we will define what a forward separable objective function is and later show that the supremum is an example of such a function. We will show that for dynamic programming problems with a forward separable objective function, augmenting the state variables allows us to use standard dynamic programming techniques to solve the problem.\\
 \\

 \begin{defn}[\cite{Envelope}]
 	The function $J(\mbf u, \mbf x)$ is said to be forward separable if there exists functions $\phi_0(x,u)$, $\phi_T(x,\phi_{T-1})$, and $\phi_i(x,u,\phi_{i-1})$ for $i=1,\cdots T-1$ such that
  	\begin{align} \label{forward_sep_def}
 	& J(\mbf u, \mbf x)=\phi_T( x(T),\phi_{T-1}[x(T-1),u(T-1),\phi_{T-2}\{....,\\ \nonumber
 	&\qquad \phi_{2}\{x(2),u(2), \phi_{1}\{x(1),u(1), \phi_{0}\{x(0),u(0)\}\}\},....,\}])\\ \nonumber
 	\end{align}
where $\phi_t\; :\; \mathbb{R}^{n} \times \mathbb{R}^{p} \times \mathbb{R}^q \to \mathbb{R}^q $, for $t=1,...,T-1$ and $\phi_{T}: \mathbb{R}^n \times \mathbb{R}^q \to \mathbb{R}$, $\phi_{0}: \mathbb{R}^n \times \mathbb{R}^p \to \mathbb{R}^q$ .
 \end{defn}

Clearly, any objective function of the form
\[
J(\mbf u, \mbf x) = \sum_{t=t_0}^{T-1} c_t(u(t),x(t)) + c_T(x(T))
\]
is forward separable using $\phi_0(x,u)=c_0(x,u)$, $\phi_T(x,\phi_{T-1}=c_T(x)+\phi_{T-1}$ and
\[
\phi_i(x,u,\phi_{i-1})=c_i(x,u)+\phi_{i-1} \quad \text{for } i=1,\cdots,T-1
\]
In addition, it can be shown that the sum of any number of forward separable functions is forward separable. For example, let $J_1(\mbf u,\mbf x)$ and $J_2(\mbf u, \mbf x)$ be forward separable with associated $\phi_i= g_i$ and $\phi_i=h_i$, respectively. Then $J_1+J_2$ is forward separable with
\begin{equation}
\phi_i(x,u,\phi_{i-1})=\bmat{\phi^1_i(x,u,\phi_{i-1})\\ \phi_i^2(x,u,\phi_{i-1})} = \bmat{g_i(x,u,\phi_{i-1}^1)\\h_i(x,u,\phi_{i-1}^2)} \label{eqn:two_forward_sep}
\end{equation}

and
\[
\phi_T(x,u,\phi_{T-1}) = g_T(x,u,\phi_{T-1}^1) + h_T(x,u,\phi_{T-1}^2).
\]
Clearly,
\[
\phi_0(x,u)=\bmat{\phi^1_0(x,u)\\ \phi_0^2(x,u)} = \bmat{g_0(x,u)\\h_0(x,u)}.
\]

 We now show that the supremum (maximum) function is forward separable.
 \begin{lem}
 	\[
 J(\mbf u, \mbf x)=\max\{\sup_{0 \le k \le T-1}\{c_k(u(k),x(k))\}, c_T(x(T))\}
 \]
 is a forward separable objective function.\\
 \end{lem}
 \begin{proof}
 	\small{ \begin{align*}
 	J(\mbf u, \mbf x)& = \max\{\sup_{0 \le k \le T-1}\{c_k(u(k),x(k))\}, c_T(x(T))\}\\
 	& = \max\{c_T(x(T)),\max\{c_{T-1}(u(T-1),x(T-1)),\cdots\\
 &\qquad \max\{..,\max\{c_1(u(1),x(1)),\max\{c_0(u(0),x(0))\}\},..\}\}
 	\end{align*} }
 \normalsize
so that
\begin{align*}
\phi_i(x,u,\phi_{i-1})&=\max(c_i(x,u),\phi_{i-1}), \quad \phi_0(x,u)=c_0(x,u),\\
\phi_T(x,\phi_{T-1}) &= \max (c_T(x),\phi_{T-1})
\end{align*}
 \end{proof}

\subsection{Forward Separable Dynamic Programming}

We may now define the class of indexed forward separable problems $H(t_0,x_0)$ so that $H$ is of class $G$, but not of class $P$ and has the form:
 \begin{align}\nonumber
 &\min_{\mathbf u, \mbf x} J_{t_0,x_0}(\mbf u, \mbf x) \\ \label{opt:forward_sep}
 &\text{subject to:  }  x(t+1)=f[x(t),u(t),t]\\ \nonumber
 & x(0)=x_0 \\ \nonumber
 & x(t) \in X \subset \mathbb{R}^n \text{ for  } t=1,..,T \\ \nonumber
 & u(t) \in U \subset \mathbb{R}^p \text{ for  } t=0,..,T-1  \nonumber
 \end{align}

where $J_{t_0,x_0}$ is forward separable with associated $\phi_i$. For every instance of a forward separable dynamic programming problem $H(t_0,x_0)$, we may associate  a new optimization problem $A(t_0,x_0)$, which is equivalent to $H(t_0,x_0)$ in a certain sense and which satisfies the principle of optimality. $A(t_0,x_0)$ is defined as follows.

\small{
 	\begin{align} \nonumber
 	& \min_{ \mathbf u} L_{t_0,x_0} (\mathbf u, \mbf x)=   z_{2}(T+1) \\ \label{augmentation}
 &	\text{subject to:  }  \bmat{z_1(t+1)\\z_2(t+1)}= \bmat{f(z_1(t),u(t)) \\ \phi_{t}(z_1(t),u(t), z_2(t)) } 1<t<T\\ \nonumber
 	& \bmat{z_1(1)\\z_2(1)}= \bmat{f(z_1(0),u(0)) \\ \phi_{0}(z_1(0),u(0)) }, \quad  \bmat{z_1(T+1)\\z_2(T+1)}= \bmat{z_1(T) \\ \phi_{T}(z_1(T),z_2(T)) }\\ \nonumber
 	& \qquad \hspace{-7mm} \bmat{z_1(0)\\z_2(0)} = \bmat{x_0\\ 0} \\ \nonumber 
 	 & z_{1}(t) \in X \text{ for  } t=1,..,T \\ \nonumber
 	 & u(t) \in U \text{ for  } t=1,..,T  \nonumber
 	\end{align} } \normalsize
Where the solution to $H(t_0,x_0)$ can be recovered as $x(t)=z_1(t)$.\\


 \begin{lem}
 	Suppose $J_{t_0,x_0}(\mbf u, \mbf x)$ is forward separable with associated $\phi_i$. Then $J^*_{t_0,x_0}=L^*_{t_0,x_0}$. Furthermore, suppose $\mbf u$ and $\mbf x$ solve $H(t_0,x_0)$  and $\mbf w$ and $\mbf z$ solve $A(t_0,x_0)$. Then $\mbf u=\mbf w$ and $x(t)=z_1(t)$ for all $t$.
 \end{lem}
 \begin{proof}
Suppose  $\mbf w$ and $\mbf z$ solve $A(t_0,x_0)$. First we show that $\mbf w$ and $\mbf z_1$ are feasible for $H(t_0,x_0)$. Clearly $w(t) \in U$ for all $t$ and if we let $\mbf u = \mbf w$ then $x(0)=x_0$ and $x(t+1)=f[x(t),u(t),t]$ for all $t$. Since likewise $z_1(0)=x_0$ and $z_1(t+1)=f[z_1(t),u(t),t]$, we have $x(t)=z_1(t)\in X$ for all $t$. Hence $\mbf u$ and $\mbf x=\mbf z_1$ are feasible for $H(t_0,x_0)$.
Likewise, if $\mbf u$ and $\mbf x$ solve $H(t_0,x_0)$, then if we let $\mbf w=\mbf u$ and $\mbf z_1=\mbf x$ and define $z_2(t+1)=\phi_{t}(z_1(t),u(t), z_2(t))$, $z_2(1)=\phi_{0}(z_1(0),u(0))$, $z_2(0)=0$, then $\mbf w$ and $\mbf z$ are feasible. Furthermore, in both cases, if we examine the objective value
  	\begin{align*}
 	&J(\mbf u, \mbf x)=\phi_T( z_1(T),\phi_{T-1}[z_1(T-1),w(T-1),\phi_{T-2}\{....,\\ \nonumber
 	&\quad \phi_{2}\{z_1(2),w(2), \phi_{1}\{z_1(1),w(1), \phi_{0}\{z_1(0),w(0)\}\}\},....,\}]).\\ \nonumber
 	\end{align*}
However, we now observe
\begin{align*}
 	&z_2(T+1)= \phi_T(z_1(T),z_2(T))\\
 	&z_2(T)= \phi_{T-1}(z_1(T-1),u(T-1),z_2(-1))\\
 	&\vdots\\
 	&z_2(2)= \phi_1(z_1(1),u(1),z_2(1))\\
 	&z_2(1)= \phi_0(z_1(0),u(0)).
\end{align*}
Hence we have

\small{
\begin{align*}
   L(\mbf w, \mbf z) & = z_1(T+1) \\
  &=\phi_T(z_1(T),u(T),\phi_{T-1}(z_1(T-1), u(T-1),\phi_{T-2}(\cdots,\\
  & \qquad  \phi_1(z_1(1),u(1),z_2(1), \phi_0(z_1(0),u(0))) \cdots )))\\
  &=J(\mbf u, \mbf x).
\end{align*} } \normalsize

Hence if $\mbf w$ and $\mbf z$ solve $A(t_0,x_0)$ with objective $L^*_{t_0,x_0}=z_2(T+1)$, then $\mbf w$ and $\mbf z_1$ solve $H(t_0,x_0)$ with objective value $J^*_{t_0,x_0}=L^*_{t_0,x_0}=z_2(T+1)$.
 \end{proof}

 \begin{prop}
 	The augmented optimization problem $A(t_0,x_0)$ in \eqref{augmentation} satisfies the Principle of Optimality and the Bellman equation \eqref{eqn:bellman}.
 \end{prop}
\begin{proof}
	$A(t_0,x_0)$ is a special case of $P(t_0,x_0)$ where $c_i=0$ for $i \neq T$.
	\end{proof}

To understand the augmented approach intuitively, we note that dynamic programming breaks a multi-period planning problem into simpler optimization problems at each stage. However, for non-separable problems, to make the correct decision at each stage we need historical data. In this context, the extra augmented state contains that part of the history necessary to make the correct decision at the present time.\\

\begin{cor}\label{cor:main}
 	$S(t_0,x_0)$ is a special case of $H(t_0,x_0)$.\\
\end{cor}
\begin{proof}
Consider the objective function from Problem $S(t_0,x_0)$ as
\[
J_{t_0,x_0}(\mbf u, \mbf x)=  \sum_{t=t_0}^{T-1}{c_{t}(x(t),u(t))} +c_{T}(x(T)) +\sup_{t_0 \le k \le T}{d_t(x({k}))}.
\]

Now this is the sum of two forward separable functions. As per the previous discussion \eqref{eqn:two_forward_sep}, then, we define
 \[
g_i(x,u,\phi^1_{i-1})=c_i(x,u)+\phi^1_{i-1} \quad \text{for } i=1,\cdots,T-1
\]
\[
g_0(x,u)=c_0(x,u), \quad g_T(x, \phi^1_{T-1})= c_T(x)+\phi^1_{T-1}
\]
and
\begin{align*}
h_i(x,u,\phi^2_{i-1})&=\max(d_i(x,u),\phi^2_{i-1}), \quad h_0(x,u)=d_0(x,u),\\
h_T(x,\phi^2_{T-1}) &= \max (d_T(x),\phi^2_{T-1})
\end{align*}
Then
\[
\phi_i(x,u,\phi_{i-1})=\bmat{\phi^1_i(x,u,\phi_{i-1})\\ \phi_i^2(x,u,\phi_{i-1})} = \bmat{g_i(x,u,\phi_{i-1}^1)\\h_i(x,u,\phi_{i-1}^2)}
\]
\[
\phi_T(x,u,\phi_{T-1}) = g_T(x,u,\phi_{T-1}^1) + h_T(x,u,\phi_{T-1}^2),
\]
and
\[
\phi_0(x,u)=\bmat{\phi^1_0(x,u)\\ \phi_0^2(x,u)} = \bmat{g_0(x,u)\\h_0(x,u)}
\]
establish forward separability of $J_{t_0,x_0}$ as per~\eqref{forward_sep_def}.
\end{proof}
 The $\phi_i$ specified in the proof of this Corollary define an instance of problem $H(t_0,x_0)$, which was shown to be equivalent to a class of optimization problems $A(t_0,x_0)$ by Lemma 4. Since problems of class $A(t_0,x_0)$ satisfy the principle of optimality, they can be solved using dynamic programming and their solution yields a solution to the original Problem $S(t_0,x_0)$. In the following section, we will apply this technique to optimal battery scheduling in the presence of demand charges.

\section{Application to the Energy Storage problem}\label{sec:battery1}
\label{sec:ProblemStatement}
In this section, we apply the augmented dynamic programming methodology to optimal scheduling of batteries in the presence of demand charges.
We first propose a simple model for the dynamics of the battery storage. We then formulate the objective function using electricity pricing plans which include demand charges. We see that the system described becomes an optimization problem of the form $H(0,e_0)$ \eqref{opt:forward_sep}.

\subsection{Battery Dynamics}
We will model the energy stored in the battery by the difference equation:
\begin{equation}
e({k+1})=\alpha (e({k})+\eta u({k})\Delta t)
\end{equation}
Where $e(k)$ denotes the energy stored in the battery at time step $k$, $\alpha$ is the bleed rate of the battery, $\eta$ is the efficiency of the battery, $u({k})$ denotes the charging/discharging $(+/-)$ at time step $k$ and $\Delta t$ is the amount of time passed between each time step. Moreover we denote the maximum charge and discharge rate by $\bar{u}$ and $\underline{u}$ respectively. Thus we have the constraint that $u({k}) \in [\underline{u},\bar{u}]:=U $ for all $k$. Similarly we also add the constraint $e(k) \in [\underline{e},\bar{e}]:=X$ for all $k$ where $\underline{e}$ and $\bar{e}$ are the capacity constraints of the battery (typically $\underline{e}=0$).

\subsection{The objective function}
Let us denote $q(k)$ to be the power supplied by the grid at time step k.
\begin{equation}
q(k)=q_{a}(k)-q_{s}(k)+u({k})
\end{equation}
where $q_{a}(k)$ is the power consumed by HVAC/appliances at time step $k$ and $q_{s}(k)$ is the power supplied by solar photovoltaics at time step $k$. For now, it is assumed that both are known apriori.\\
\\
To define the cost of electricity we divide the day $t \in [0,T]$ into on-peak and off-peak periods. We define an off peak period starting from 12am till $t_{\text{on}}$ and $t_{\text{off}}$ till 12am. We define an on-peak period between $t_{\text{on}}$ till $t_{\text{off}}$. The Time-of-Use (TOU, \$ per kWh) electricity cost during on-peak and off-peak is denoted by $p_{\text{on}}$ and $p_{\text{off}}$ respectively. We further simplify this as $p_k = p_{on}$ if $ k \in T_{on}$ and $p_k = p_{off}$ if $ k \in T_{off}$ where $T_{on}$ and $T_{off}$ are the on-peak and off-peak hours, respectively. These TOU charges define the first part of the objective function as:\\
\begin{align*}
J_{\text{E}}(\mbf u, \mbf e)&=p_{\text{off}} \sum_{k=0}^{t_{\text{on}}-1}{q(k)\Delta t} + p_{\text{on}} \sum_{k=t_{\text{on}}}^{t_{\text{off}}-1}{q(k)\Delta t} + p_{\text{off}} \sum_{k=t_{\text{off}}}^{T}{q(k)\Delta t}\\
&=\sum_{k \in [0,T]} p_k (q_{a}(k)-q_{s}(k)+u({k}))\Delta t\\
&=\sum_{k \in [0,T]} p_k (q_{a}(k)-q_{s}(k))\Delta t+\sum_{k \in [0,T]} p_k u({k})\Delta t
\end{align*}
Where the daily terminal timestep is $T=24/\Delta t$. Clearly, only the second term in this objective function is significant for the purposes of optimization.\\
\\
We also include a demand charge, which is a cost proportional to the maximum rate of power taken from the grid during on-peak times. This cost is determined by $p_{d}$ which is the price in \$ per kW. Thus it follows the demand charge will be:\\
\begin{align*}
J_{D}(\mbf u, \mbf e)&=p_{d} \sup_{k \in \{t_{\text{on}},....,t_{\text{off}}-1\}}{q(k)}\\
&p_{d} \sup_{k \in \{t_{\text{on}},....,t_{\text{off}}-1\}}\{q_{a}(k)-q_{s}(k)+u({k})\}
\end{align*}

\subsection{24 hr Optimal Residential Battery Storage Problem}
We may now define the problem of optimal battery scheduling in the presence of demand and Time-of-Use charges, denoted $D(0,e_0)$.
\begin{align*}
& \min_{\mbf u, \mbf e}\{J_{E}(\mbf u, \mbf e)+J_{D}(\mbf u, \mbf e)\} \text{     subject to}\\
& e(k+1)=\alpha (e(k)+\eta u({k})\Delta t)\text { for } k=0,...,T\\
& e(k)\in X \text { for } k=0,...,T\\
& u(k) \in U \text { for } k=0,...,T\\
& e_{0}=e0
\end{align*}
Where recall $U:=[\underline{u},\bar{u}]$ and $X:=[\underline{e},\bar{e}]$.
\begin{prop}
  Problem $D(0,e_0)$ is a special case of $S(t_0,x_0)$
\end{prop}
\begin{proof}
  Let $c_i= p_i (q_{a}(i)-q_{s}(i)+u({i}))\Delta t$
\[
d_i=\begin{cases}
p_d(q_{a}(k)-q_{s}(k)+u_{k})& k \in T_{on}\\
0 & \text{otherwise.}
\end{cases}
\]
\end{proof}
We conclude that our algorithmic approach to forward separable dynamic programming can be applied to this problem as per Corollary~\ref{cor:main}. That is, it can be represented as an augmented dynamic programming problem of Form $A(t_0,x_0)$.

\section{Numerical Implementation}
To illustrate our approach to generalized dynamic programming, we use solar and usage data obtained by local utility Salt River Project in Tempe, AZ. We also use pricing data from SRP and battery data obtained for the Tesla Powerwall. As is standard practice, for implementation, we used a discrete input and state space. The results of the simulation are shown in Fig.~\ref{fig:deterministic}. These results show a slight improvement in accuracy over results obtained based on the approach to a similar problem in~\cite{Thermostat} (approximately \$0.98 savings).
\small{
\begin{table}
	\centering
	\caption{List of constant values (prices correspond to Salt River Project E21 price plan)}
	\label{tab:constants}
	\begin{tabular}{|l|l||l|l|l}
		\cline{1-4}
		Constant  & Value       &  Constant  & Value &  \\ \cline{1-4}
		$\alpha$  & 0.999791667 (W/h) &  $t_{\text{off}}$  & 41  &  \\ \cline{1-4}
		$\eta$    & 0.92  (\%)      &  $p_{\text{on}}$    & $0.0633 \times 10^{-3}$ (\$/KWh) &  \\ \cline{1-4}
		$\bar{u}$ & 4000 (Wh)       &  $p_{\text{off}}$ & $0.0423 \times 10^{-3}$ (\$/KWh) &  \\ \cline{1-4}
		$\underline{u}$  & -4000 (Wh) &  $p_{\text{d}}$  & 3.364 (\$/KWh) &  \\ \cline{1-4}
		$\bar{e}$    & 8000   (Wh)     &  $\Delta t$    & 0.5 (h) &  \\ \cline{1-4}
		$t_{\text{on}}$ & 27        &  &  &  \\ \cline{1-4}
	\end{tabular}
\end{table}
}
\normalsize

\begin{figure}
	\includegraphics[scale=0.2]{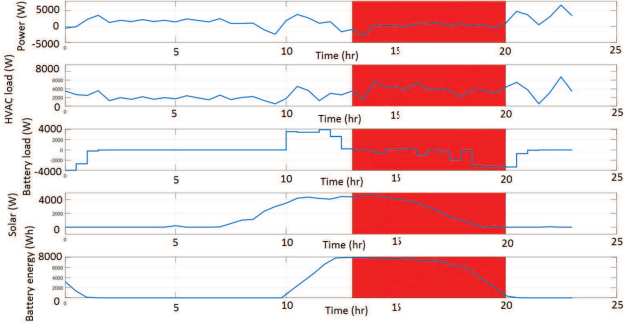}
	\label{fig:stochastic}
	\caption{The trajectory the algorithm produces for randomly generated stochastic solar data. The supremun of the power is 1.66(kw) and the cost is \$64.9889.}
\end{figure}

\begin{figure}
	\includegraphics[scale=0.2]{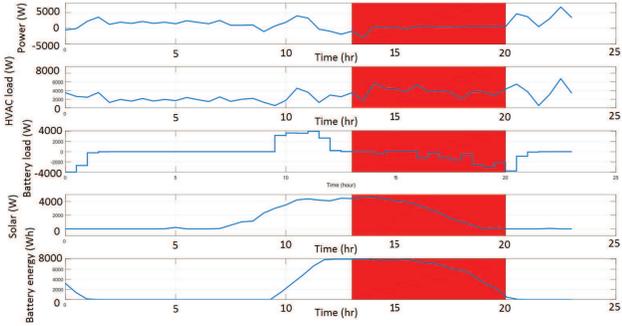}
	\label{fig:deterministic}
	\caption{The trajectory the algorithm produces for deterministic solar data. The supremun of the power is 0.7033(kw) and the cost is \$46.389.}
\end{figure}

\section{Using a Stochastic Model} \label{sec:SDP}
To show that this approach can also be extended to stochastic dynamic programming and to evaluate the effect of stochastic uncertainty on battery scheduling, we identified a Gauss-Markov model of solar generation based on SRP data. We then used a trivial extension of problem $A(t_0,s_0)$ to forward separable dynamic programming with stochastic disturbances.

\subsection{Solar Generation Model}
Our approach to modeling the dynamics of solar generation for a given subset of data is to use solar irradiance directly as a primary variable along with other possible correlated variables such as temperature or 2-hr Deltas in pressure. Specifically, we take time-series data of these quantities, denoted $\bold{W}(t)$ and normalize this data as\vspace{-2mm}
\begin{align*}
w_{i}(t)=\dfrac{W_{i}(t)-\mu_{i}(t)}{\sigma_{i}(t)}\vspace{-2mm}
\end{align*}
Where $\mu_{i}(t)$ is the average historic and clear-sky mean of the variable $W_{i}$ at time step $t$ and $\sigma_{i}(t)$ is the standard deviation of variable $W_{i}$ at time step $t$.\\
The generating process is then given by:\vspace{-2mm}
\begin{align*}
& \bold{w}(t)=A\bold{w}(t-1)+B\bold{\epsilon}(t-1) \text{  for  } t=1,..,T\\
& \text{where } \bold{w}(t) \in \mathbb{R}^3, \bold{w}(0)=\bold{0}\\
& \epsilon(t) \sim \mathbb{N}(\bold{0},\Sigma  ) \text{ , } \Sigma_{i,j}=\delta_{i,j}:=\begin{cases}1 & i=j\\0& i \neq j\end{cases}\vspace{-2mm}
\end{align*}
Where the matrices $A$ and $B$ are chosen to preserve the lag 0 and lag 1 cross-correlations seen in the collected data. Specifically, we can compute these matrices as (\cite{Solar})\vspace{-2mm}
\[
 A=M_{1}M_{0}^{-1}\qquad \qquad \qquad  BB^{T}=M_{0}-M_{1}M_{0}^{-1}M_{1}^{T}
\]
Where $M_{i}$ is the i-lag cross correlation matrix. So $(M_{i})_{m,n}=\rho_{i}(m,n)$ where $\rho_{i}(m,n)$ is the cross-correlation coefficient between variables m and n with variable n lagged by i time steps. Then, adding back in the mean and deviation, we obtain the power supplied by solar at time step $k$ as\vspace{-2mm}
\[
q_{s}(k)=w_{1}(k) \sigma_{1}(k) + \mu_{1}(k)\vspace{-2mm}
\]
\subsection{Augmented Stochastic Dynamic Programming}
We now define a class of Stochastic Dynamic Programming problems, $T(t_0,x_0)$ of the form
 \begin{align}
  &\min_{\mathbf u, \mathbf x}  J_{t_0,x_0}(\mathbf u, \mathbf x) = \mbb E\left( \sum_{t=t_0}^{T-1}{c_{t}(x(t),u(t))} + c_{T}({x(T)})  \right) \label{eqn:DP}\\ \nonumber
  &\text{subject to:  }  x(t+1)=f[x(t),u(t),t;v(t)] \text{,  given  } x(t_0)=x_0  \\ \nonumber
      & x(t) \in X \text{ for  } t=t_{0}+1,..,T \\ \nonumber
      & u(t) \in U \text{ for  } t=t_0,..,T-1  \nonumber\\
      & v(t) \sim \mbb N(\mbf 0,\Sigma),\qquad  \Sigma_{i,j}=\delta_{i,j}
 \end{align}
As shown in~\cite{Bellman}, A stochastic version of the Bellman Equation can be used to solve Stochastic Dynamic Programming problems of the Form $T(t_0,x_0)$. Specifically, suppose that $F$ satisfies $F(x,T)=c_T(x)$ and
\begin{align} \label{Bellman_stochastic}
F(x,t)=  \inf_{u}\{& c_t(x,u )\\ \nonumber
&\qquad + \mathbb{E}_v[F(f[x,u,t;v],t+1) \;|\; x,u\; ]\}. \nonumber
\end{align}
Then for problem $T(t_0,x_0)$, $F(x,t)= J_{t,x}^*$ and $\pi^*(x)=\inf_{u}\{ c_t(x,u )+ \mathbb{E}_v[F(f[x,u,t;v],t+1)  ]\}$ defines an optimal policy.

\noindent \textbf{Stochastic Battery Scheduling} We now modify Problem $D(0,e_0)$ to give a stochastic version of the battery scheduling problem
\begin{align*}
& \min_{\mbf u, \mbf e} \mbb E \left(\{J_{E}(\mbf u, \mbf e)+J_{D}(\mbf u, \mbf e)\}\right) \text{     subject to}\\
& e(k+1)=\alpha (e(k)+\eta u({k})\Delta t)\text { for } k=0,...,T\\
& w(k+1)=Aw(k)+B\epsilon(k)\text { for } k=0,...,T\\
& e(k)\in X \text { for } k=0,...,T\\
& u(k) \in U \text { for } k=0,...,T\\
& e_{0}=e0,\qquad \epsilon(k) \sim \mathbb{N}(\bold{0},\Sigma  )
\end{align*}

To solve the stochastic version of $D(0,e_0)$, we augment to obtain a stochastic version of Problem $A(t_0,x_0)$, which is a special case of $T(t_0,x_0)$, which then admits a solution using the stochastic version of Bellman's equation.

\subsection{Implementation of the Stochastic Algorithm}
The primary challenge with implementation is computing the expectation in Bellman's equation \eqref{Bellman_stochastic}. Specifically, if $\phi$ is the pdf of $\bold{v}(t)$, we must compute

\small{
\begin{align*}
\mathbb{E}_v[F(f[x,u,t;v],t+1) \;|\; x,u\; ]\}= \int{F(f(x,u,t;v),t+1) \phi(v) dv }
\end{align*}
}
\normalsize
To numerically integrate this function, we discretize $u$, $x$ and $v$ so that the integral becomes a sum where $\phi(v_i)$ is a weighted sample from the normal distribution. The results of this algorithm are shown in Figure~\ref{fig:stochastic} using the parameter values from Table~\ref{tab:constants}. The solar data generated from this run were then used as input to the deterministic algorithm in order to compare performance. As expected, the deterministic case performs better than the stochastic case.

\section{Conclusion}
In this paper we have proposed a generalized formulation of the dynamic programming problem and shown that if the objective function is forward separable, these problems may be solved using an equivalent augmented dynamic programming approach. Furthermore, we have shown that the problem of optimal scheduling of battery storage in the presence of combined demand and time-of-use charges is a special case of this class of forward separable dynamic programming problems. We have further extended these results to stochastic dynamic programming with a forward separable objective. The proposed algorithms were demonstrated on a battery scheduling problem using first a deterministic and then Gauss-Markov model for solar generation and load.
%

\bibliographystyle{IEEEtran}
\bibliography{bibliography,bibliography2,bibliography3}
\end{document}